\documentclass[a4paper,11pt,reqno]{amsart}
\usepackage[margin=1.2in]{geometry}
\usepackage[utf8]{inputenc}
\usepackage[T1]{fontenc}
\usepackage[english]{babel}
\usepackage{tgpagella}
\usepackage[colorlinks=true, pdfstartview=FitV, linkcolor=blue, citecolor=blue, urlcolor=blue]{hyperref}
\usepackage[capitalise, noabbrev, nameinlink]{cleveref}
\usepackage{amscd}
\usepackage{csquotes}
\usepackage{fancyhdr}
\usepackage{multirow}
\usepackage{cleveref}
\usepackage{comment}
\usepackage{soul}
\usepackage{amsthm, amsfonts, amssymb, amsmath, latexsym, array,color}
\usepackage{amscd}
\usepackage[all]{xy}
\usepackage{pstricks,graphicx}
\usepackage{stmaryrd}
\usepackage{hyperref,mdwlist,mathrsfs}
\usepackage{tikz}
\usepackage{pgf,tikz}
\usepackage{tikz-cd,rotating}
\usepackage{pdflscape}
\usepackage{array}
\usetikzlibrary{arrows}
\usepackage{todonotes}

\newcommand {\PP}{\mathbb{P}}

\usepackage{color}

\usepackage{amssymb}
\usepackage{amsmath}
\usepackage{amsthm}
\usepackage{paralist}

\newcommand{\HF}{\ensuremath{\mathrm{HF}}}

\newtheorem{theorem}{Theorem}[section]
\newtheorem*{theorem*}{Main Theorem}
\newtheorem{lemma}[theorem]{Lemma}

\newtheorem{proposition}[theorem]{Proposition}
\newtheorem{corollary}[theorem]{Corollary}

\newtheorem*{proposition*}{Proposition}
\theoremstyle{definition}
\newtheorem*{Oprobl}{Open Problems}
\newtheorem{definition}[theorem]{Definition}
\newtheorem{example}[theorem]{Example}

\DeclareMathOperator{\Ann}{Ann}

\DeclareMathOperator{\hess}{hess}
\DeclareMathOperator{\Hess}{Hess}
\DeclareMathOperator{\Hom}{Hom}

\DeclareMathOperator{\Soc}{Soc}
\DeclareMathOperator{\reg}{reg}

\DeclareMathOperator{\init}{in}

\newcommand{\qand}{\quad \mbox{and} \quad}

\newcommand{\qwith}{\quad \mbox{with} \quad}

\numberwithin{equation}{section}

\usepackage{enumitem}

\newcommand{\AS}[1]{
  {\color{magenta} Alexandra: #1}}
\newcommand{\na}[1]{
  {\color{blue} Nasrin: #1}}

\title[New families of AG algebras satisfying WLP]{New families of Artinian Gorenstein algebras with the weak Lefschetz property}

\author[]{Nasrin Altafi}
\address{Nasrin Altafi: Department of Mathematics, KTH Royal Institute of Technology, S-100 44 Stockholm, Sweden and Department of Mathematics, Queen's University, 505 Jeffery Hall, University Avenue, Queen's University, Kingston, Ontario, Canada K7L 3N6}
\email{nar3@queensu.ca}

 \author[]{Rodica Dinu}
 \address{Rodica Dinu: University of Konstanz, Fachbereich Mathematik und Statistik, Fach D 197 D-78457 Konstanz, Germany, and Simion Stoilow Institute of Mathematics of the Romanian Academy, Calea Grivitei 21, 010702, Bucharest, Romania}
  \email{rodica.dinu@uni-konstanz.de}

\author[]{Shreedevi K. Masuti}
\address{Shreedevi K. Masuti: Department of Mathematics, Indian Institute of Technology Dharwad, Permanent Campus, Chikkamalligwad,  Dharwad - 580011, Karnataka, India}
\email{shreedevi@iitdh.ac.in}

  \author[]{Rosa M.\ Mir\'o-Roig}
  \address{Rosa Maria Mir\'o-Roig: Facultat de
  Matem\`atiques i Inform\`atica, Universitat de Barcelona, Gran Via des les
  Corts Catalanes 585, 08007 Barcelona, Spain} \email{miro@ub.edu,  ORCID 0000-0003-1375-6547}

\author[]{Alexandra Seceleanu}
\address{Alexandra Seceleanu: Department of
  Mathematics, University of Nebraska-Lincoln, 203 Avery Hall, Lincoln, NE 68588, USA}
\email{aseceleanu@unl.edu}

\author[]{Nelly Villamizar}
\address{Nelly Villamizar: Department of Mathematics, Swansea University, Fabian Way, SA1 8EN, Swansea, UK
}
\email{n.y.villamizar@swansea.ac.uk}

\thanks{\hspace{-15pt}Altafi was supported by Swedish Research Council grant VR2021-00472, Dinu was supported by the Alexander von Humboldt Foundation and DFG grant nr 467575307,  Masuti is supported by CRG grant CRG/2022/007572 and MATRICS grant MTR/2022/000816 funded by ANRF,  Govt. of India, Mir\'o-Roig was partially supported by the grant PID2020-113674GB-I00,
Seceleanu was partially supported by NSF DMS–2101225 and DMS–2401482, Villamizar was partially supported by the EPSRC New Investigator Award EP/V012835/1.}

\theoremstyle{definition}

\begin{document}

\begin{abstract} We construct new 
 families of Artinian Gorenstein  graded $K$-algebras of arbitrary codimension having binomial Macaulay dual generators and satisfying the weak or the strong Lefschetz property. This is a companion paper to \cite{ADFMMSV}, which studies  codimension three algebras having binomial Macaulay dual generators in great depth, establishing in particular that they enjoy the strong Lefschetz property.
\end{abstract}

\maketitle

\date{}

\setcounter{tocdepth}{1}
\tableofcontents

\section{Introduction}


The main objective of this paper is to advance the understanding of a challenging question in commutative algebra: under what conditions does a graded Artinian Gorenstein (AG) algebra over a field $K$ satisfy the weak Lefschetz property (WLP) or the strong Lefschetz property (SLP)? These properties play a crucial role in the study of Artinian algebras, influencing their algebraic structure and leading to geometric and combinatorial applications; see \cite{LefscetzBook, MNtour, JMR} for overviews of the topic and \cite{stanley2, HMNW} for  crucial advances. 

Via Macaulay-Matlis duality, an AG graded \( K \)-algebra corresponds to a homogeneous polynomial, called the Macaulay dual generator. When this generator is a monomial, the AG algebra is a quotient of the polynomial ring by a monomial complete intersection. In this case, previous studies by Stanley  \cite{stanley2} and Watanabe \cite{watanabe} have shown that if $K$ has characteristic zero the resulting AG algebra satisfies the WLP and even the SLP. Along with this, AG algebras having monomial Macaulay dual generator satisfy several other desirable properties which we list below:
\begin{inparaenum}
\item The defining ideal is a monomial complete intersection;
    \item The minimal generators of the monomial complete intersection form a Gröbner basis;
    \item The minimal free resolutions for (monomial) complete intersections are given by the Koszul complex;
    \item (Monomial) complete intersections can be obtained by a construction termed doubling from 0-dimensional schemes in $\PP^{n-1}$.
\end{inparaenum}

The present paper and the companion paper \cite{ADFMMSV} investigate to what extent similar results hold for AG algebras having binomial Macaulay dual generator. In particular, given the list of properties enumerated above we ask the following questions motivating our investigations:

\begin{Oprobl}\label{oprobl}

\begin{enumerate}
\item {\em Characterization of complete intersections:} Under what conditions can an AG algebra with a binomial Macaulay dual generator be classified as a complete intersection? This classification could help simplify the study of these algebras and provide a clearer structural picture. It is a step towards the larger goal of characterizing Artinian complete intersections based on their Macaulay dual generators. 
\item {\em Strong and weak Lefschetz properties:} While codimension-three AG algebras with binomial Macaulay dual  generators are shown in \cite{ADFMMSV} to satisfy both WLP and SLP, determining broader conditions that guarantee these properties in higher codimensions remains an open problem.
\item {\em Bounds on the Sperner number and flat h-vectors:} The Sperner number, which reflects the maximal value of the Hilbert function of an AG algebra, and the flatness of its h-vector (i.e., the number of equal entries corresponding to the Sperner number) are critical aspects that influence the structure of the algebra and allow to transfer the Lefschetz properties from an AG algebra to some of its quotients.
\item {\em Doubling construction:} What criteria allow these algebras to be viewed as doublings, i.e., constructed as extensions of certain one-dimensional Cohen-Macaulay algebras? Understanding the doubling properties may provide a pathway to building higher-dimensional AG algebras having specific properties of interest. It further relates the properties of AG algebras to those of non-reduced point schemes in projective space.
\item {\em Minimal Generators and Free Resolutions:} Determining bounds on the minimal number of generators required to define these algebras and describing their minimal free resolutions are fundamental to understanding the presentation and homological properties of AG algebras with binomial Macaulay dual  generators. Understanding which algebras arise as doublings  may provide a pathway to determining their minimal free resolutions.
\end{enumerate}
\end{Oprobl}

 It is shown in \cite{ADFMMSV} that codimension three AG algebras having binomial Macaulay dual generators retain most of the  properties of monomial complete intersections, satisfying in particular both WLP and SLP. However, extending these results to higher codimensions has proven to be a complex task, as some binomially generated AG algebras in codimension four or greater fail to meet the WLP under similar conditions; see \cite[Example 3.5]{ADFMMSV}. In the current paper, we focus on identifying new families of AG algebras in higher codimensions that do satisfy the WLP or SLP, as well as investigating their structural characteristics.

Our main results can be summarized as follows:

\begin{theorem*}
    Let $K$ be a field of characteristic zero and let $a_i, b_i,c_i, m$ be non-negative integers. Let $A_F$ denote the AG algebra having Macaulay dual generator $F$. 
    \begin{enumerate}
    \item If $F=X_1^{a_1}\cdots X_n^{a_n}(X_1^{b_1}\cdots X_r^{b_r}-X_{r+1}^{b_{r+1}}\cdots X_n^{b_n})$ for some $1\leq r\leq n$ is homogeneous, i.\,e.\, $ b_1+\cdots+b_r= b_{r+1}+\cdots+b_n$, and the inequality $a_1> a_2+\cdots +a_n + b_{r+1}+\cdots +b_n$ holds, then $A_F$ satisfies the WLP.
    \item If $F=X_i^{a_i}(X_1^{b_1}-X_2^{b_2}X_3^{b_3}\cdots X_n^{b_n})$ for some $1\leq i\leq n$ with $\sum_{i=2}^nb_i=b_1$, then $A_F$  has the WLP.
    \item If $F=X_1^{c_1}X_2^{c_2}X_3^{c_3}\cdots X_n^{c_n}(X_1^m-X_2^m)$, then $A_F$ has the SLP.
    \item  If $F=X_1^{c_1}X_2^{c_2}X_3^{c_3}\cdots X_n^{c_n}(m_1-m_2)$ where $m_1$ and $m_2$ are monomials of the same degree in 3 variables $X_{i}, X_j,X_r$, then $A_F$ has the SLP.
    \end{enumerate}
\end{theorem*}

In \cref{s: background}, we discuss the tools necessary for our analysis, which include higher Hessian matrices as introduced by Watanabe \cite{w1} and Gondim--Zappal\`a \cite{GZ} and a criterion for WLP from \cite{ADFMMSV} originating in work concerning Lefschetz properties of AG algebras that are connected sums  by McDaniel--Iarrobino--Seceleanu  \cite{IMS}. Making use of these tools, we prove our main results in \cref{s: main}. To conclude, we re-emphasize the open problems proposed above in \cref{s: open}.  They reflect challenges and potential directions for future research applicable to all AG algebras, which we are optimistic can be settled in the setting proposed in this paper by leveraging the  binomial nature of the Macaulay dual generator. 

Although many of these questions are resolved in the codimension three case, they remain largely open in higher codimensions, with the only partial results available to date presented in this paper.  Our work lays the groundwork for advancing the themes proposed above and opens pathways for further investigation into the structural nuances of AG algebras with binomial Macaulay dual generators. Addressing these questions will deepen our understanding of AG algebras at large.

 \textbf{Acknowledgement.} The project started at the  meeting ``Women in Commutative Algebra II '' (WICA II)  at CIRM Trento, Italy, October 16--20, 2023.  The authors would like to thank the CIRM and the organizers for the invitation and financial support through the NSF grant  DMS–2324929.

\section{Background} \label{s: background}

Throughout this paper  $K$ will be a field of characteristic zero.
Given a standard graded Artinian $K$-algebra $A=R/I$ where $R=K[x_1,\dots,x_n]$ and $I$ is a homogeneous ideal of $R$,
we denote by $\HF_A\colon \mathbb{Z} \longrightarrow \mathbb{Z}$ with $\HF_A(j)=\dim _K[A]_j=\dim _K[R/I]_j$
the Hilbert function of $A$ and $H_A(t)=\sum _{i} \dim_K[A]_it^{i}$ its Hilbert series. Since $A$ is Artinian, its Hilbert function is
encoded in its \emph{$h$-vector} $h=(h_0,\dots ,h_d)$ where $h_i=\HF_A(i)>0$ and $d$ is the largest index with this property. The integer $d$ is called the \emph{socle degree of} $A$ or the {\em regularity} of $A$ and denoted $\reg(A)$.

 \subsection{Gorenstein algebras and Macaulay duality}

A standard graded Artinian $K$-algebra $A=R/I$  with socle degree $d$ is said to be  {\em Gorenstein} if its
socle $(0:m_A)$ is a one dimensional $K$-vector space. 
We sometimes refer to the number of variables of $R$ as the {\it codimension} of $A$.

We  recall the construction of the standard graded Artinian Gorenstein algebra $A_F$ with {\em Macaulay dual generator} a given form $F\in S=K[X_1,\dots,X_N]$. We denote by $R=K[x_1,\ldots,x_N]$ the ring of differential operators acting on the polynomial ring $S$, i.\,e.\ $x_i=\frac{\partial}{\partial X_i}$. Therefore $R$ acts on $S$ by differentiation. Given polynomials $p\in R$ and $G\in S$ we will denote by ${p\circ G}$ the differential operator $p\left(\frac{\partial}{\partial X_1}, \ldots, \frac{\partial}{\partial X_n}\right)$ applied to $G$. For a homogeneous polynomial $F \in S$, we define
\[
\Ann_R(F):=\{p\in R \mid p\circ F=0\}\subset R, \quad \text{ and }  \quad A_F=R/\Ann_R(F).
\]
The ring $A_F$ is a standard graded Artinian Gorenstein $K$-algebra (AG algebra, for short) and $F$ is called its {\em Macaulay dual generator}. It is worthwhile to point out that every standard graded Artinian Gorenstein $K$-algebra is of the form $A_F$ for some homogeneous polynomial $F$, in view of Macaulay's ``double annihilator theorem'', see for instance \cite[Lemma 2.12]{IK}. Moreover the degree of the Macaulay dual generator $F$ is equal to the re\-gu\-larity of $A_F$.


Next, we single out the largest value of the Hilbert function and the degree range where it is attained. If this degree range is relatively wide, the Hilbert function has a flat.
\begin{definition}
    Let $A$ be an AG algebra. We define the {\em Sperner number}  $S_A$ of $A$ as the largest value of its Hilbert function.
    
    Let $h=(h_0,\dots ,h_d)$ be the h-vector of $A$. We set $NS_A=|\{h_i \mid h_i=S_A\}|$
    and we will say that the h-vector of $A$ has a \it{flat} if $NS_A\ge 3$.
\end{definition}

\subsection{Lefschetz properties and Hessian matrices}
\begin{definition}
Let $A=R/I$ be a graded Artinian $K$-algebra. We say that $A$ has the {\em weak Lefschetz property} (WLP, for short)
if there is a linear form $\ell \in [A]_1$ such that, for all
integers $i\ge0$, the multiplication map
\[
\times \ell: [A]_{i}  \longrightarrow  [A]_{i+1}
\]
has maximal rank.
 In this case, the linear form $\ell$ is called a weak Lefschetz
element of $A$. If for the general form $\ell \in [A]_1$ and for an integer $j$ the
map $\times \ell:[A]_{j-1}  \longrightarrow  [A]_{j}$ does not have maximal rank, we will say that the ideal $I$ fails the WLP in
degree $j$.

$A$ has the {\em strong Lefschetz property} (SLP, for short) if there is a linear form $\ell \in [A]_1$ such that, for all
integers $i\ge0$ and $k\ge 1$, the multiplication map
\[
\times \ell^k: [A]_{i}  \longrightarrow  [A]_{i+k}
\]
has maximal rank.  Such an element $\ell$ is called a strong Lefschetz element for $A$.
\end{definition}

To determine whether an Artinian standard graded $K$-algebra $A$ has the WLP or SLP seems a simple problem of linear algebra, but it has proven to be extremely elusive and  much more work on this topic remains to be done, see \cite{JMR}.

We end this preliminary subsection with results due to Watanabe \cite{w1} and Gondim-Zappal\`a \cite{GZ}  which establish a useful connection between the failure of Lefschetz properties and the vanishing of higher order Hessians.

\begin{definition}
Let $F\in S = K[X_1, \dots  , X_n]$ be a homogeneous polynomial and let $ A_F = R/\Ann_R (F)$ be the AG algebra with Macaulay dual generator $F$. Fix  ordered
$K$-basis $$\mathcal{B} = \{w_i \mid 1\le i \le h_t:=\dim_K A_t \} \subset A_t \text{ and}$$
$$\mathcal{B'} = \{u_j \mid 1\le j \le h_s:=\dim_K A_s \} \subset A_s.$$
The $(t,s)$-th mixed (relative) {\em Hessian matrix} of $F$ with respect to $\mathcal{B}, \mathcal{B'}$ is  the $h_t \times h_s$ matrix:
$$
\Hess_F^{(s,t)}=((w_iu_j)\circ F)_{i,j}.$$
The  $(t,s)$-th mixed {\em Hessian of } $F$ {\em with respect to} $\mathcal{B}, \mathcal{B'}$ is
$$ \hess _F^{(s,t)}=\det (\Hess _F^{(s,t)}).
$$
Ween $s=t$ and $\mathcal{B}=\mathcal{B'}$ the notions above become 
the $t$-th (relative) {\em Hessian matrix} of $F$ with respect to $\mathcal{B}$, namely the $h_t \times h_t$ matrix:
$$
\Hess_F^t=((w_iw_j)\circ F)_{i,j}$$
and the $t$-th {\em Hessian of } $F$ {\em with respect to} $\mathcal{B}$, respectively, 
$$ \hess _F^t=\det (\Hess _F^t).
$$
\end{definition}

The 0-th Hessian is the polynomial $F$ itself and, in the case $\dim_K A_1=n$, the 1-st Hessian, with respect to the standard basis given by the variables $x_1, \ldots, x_n$, is the classical Hessian $\Hess_F^1=(\frac{\partial^2 F}{\partial X_i\partial X_j})_{i,j}$. The definition of
Hessians of order $t$ depends on the choice of a basis of $A_t$ but  the vanishing of the $t$-th Hessian is independent of this choice.
We denote the $t$-th Hessian evaluated at the point $(a_1,\dots ,a_n)$ by $\hess _F^t(a_1,\dots, a_n)$.

\begin{theorem}[{\cite[Theorem 4]{w1}, \cite[Theorem 3.1]{MW}, \cite[Theorem 2.4]{GZ}}] \label{watanabe}
Let $F\in S = K[X_1, \dots  , X_n]$ be a homogeneous polynomial of degree $d$ and let $ A_F = R/\Ann_R (F)$ be the
associated AG algebra. Consider a linear form $\ell=a_1x_1+\cdots +a_nx_n\in A_1$.
\begin{enumerate}
    \item  The multiplication map $\times \ell^{d-2t}:A_t\to A_{d-t}$ has full rank if and only if the $t$-th Hessian $\hess _F^t(a_1,\dots, a_n)\ne 0$. In particular, $\ell$ is a strong Lefschetz element of $A$ if and only if $\hess _F^t(a_1,\dots, a_n)\ne 0$ for $t=1,\dots,\lfloor d/2\rfloor$.
    \item  The multiplication map $\times \ell^{d-s-t}:A_s\to A_{d-t}$ has full rank if and only if the $(s,t)$-th mixed Hessian $\hess _F^{(s,t)}(a_1,\dots, a_n)\ne 0$. 
\end{enumerate}
\end{theorem}

The theorem above can be used to establish the WLP as well as the SLP. Due to their duality properties, the WLP for AG algebras is determined by a single map.

\begin{proposition}[{\cite[Proposition 2.1]{MMN}}]\label{mid map}
    Let $A$ be an AG algebra with $\reg(A)=d$. Then the following are equivalent 
    \begin{itemize}
        \item[(i)] $\ell\in A_1$ is a Weak Lefschetz element on $A$ 
        \item[(ii)] the map $\times \ell: A_{\lfloor\frac{d-1}{2}\rfloor} \to A_{\lfloor\frac{d-1}{2}\rfloor+1}$ has maximum rank (is injective) 
        \item[(iii)] the map $\times \ell: A_{\lfloor\frac{d}{2}\rfloor} \to A_{\lfloor\frac{d}{2}\rfloor+1}$ has maximum rank (is surjective).
        \end{itemize}
\end{proposition}

As a consequence of \cref{watanabe} and \cref{mid map}, the WLP is determined by the non-vanishing of a single mixed Hessian determinant. 

\begin{corollary}\label{cor: hessian}
    Let $F$ be a homogeneous polynomial of degree $d$ and let $ A_F = R/\Ann_R (F)$ be the associated AG algebra. Set $t=\lfloor\frac{d}{2}\rfloor$. A linear form $\ell=a_1x_1+\cdots +a_nx_n\in A_1$ is a Weak Lefschetz element on $A$ if and only if 
     the (mixed) Hessian  $\hess _F^{(d-t-1,t)}(a_1,\dots, a_n)\ne 0$.
\end{corollary}

\section{Algebras with binomial Macaulay dual generator that satisfy the WLP} \label{s: main}

Throughout this section, we fix a binomial $F=m_1-m_2\in K[X_1,\dots ,X_n]$ with $m_1, m_2$ monomials. After reordering the variables, if necessary, any binomial can be factored as 
\[
F=X_1^{a_1}\cdots X_n^{a_n}(X_1^{b_1}\cdots X_r^{b_r}-X_{r+1}^{b_{r+1}}\cdots X_n^{b_n}), \quad\text{ where } 1\le r\le n-1 \text{ and}
\]
\begin{align}\label{eq: binomial}
m_1 & =  X_1^{a_1+b_1}\cdots X_r^{a_r+b_r} X_{r+1}^{a_{r+1}}\cdots X_n^{a_n}, \nonumber \\
 m_2 & = X_1^{a_1}\cdots X_r^{a_r} X_{r+1}^{a_{r+1}+b_{r+1}}\cdots X_n^{a_n+b_n}, \nonumber \\
 g &=\gcd(m_1,m_2)  =   X_1^{a_1}\cdots X_n^{a_n} \nonumber \\
 d &= \sum _{i=1}^na_{i} + \sum_{i=1}^rb_{i} =  \sum _{i=1}^na_{i} + \sum_{i=r+1}^n b_{i}.
 \end{align}
 Our goal will be to describe the main features of the Artinian Gorenstein algebra $A_F$ with binomial Macaulay dual generator $F$ in terms of $a_i$ and $b_i$. 


\subsection{A criterion for WLP}
In \cite{ADFMMSV}, the authors prove that  codimension 3 AG algebras $A_F$ with binomial Macaulay dual generator $F$ enjoys the SLP, but even the WLP can fail for AG algebras of higher codimension with binomial Macaulay dual generator. The same paper gives a useful general result which we employ in our present work.

\begin{theorem}[{\cite[Theorem 3.4]{ADFMMSV}}]\label{thm: conn sum criterion}
  If $F=m_1-m_2$ is a homogeneous binomial such that $g=\gcd(m_1,m_2)$ satisfies  $\deg (g)<\left \lfloor\frac{\deg (F)-1}{2} \right \rfloor$, then $A_F$ has the WLP.
\end{theorem}


The result above is the best possible in the sense that there are examples of Artinian Gorenstein algebras $A_F$ with $\deg (g)\ge \left \lfloor\frac{\deg (F)-1}{2} \right \rfloor$ failing the WLP  \cite[Example 3.5]{ADFMMSV}. Our aim is to describe new families of binomials which in spite of verifying the inequality 
$\deg (g)\ge \left \lfloor\frac{\deg (F)-1}{2} \right \rfloor$
 do satisfy the WLP.

 The following lemma will be useful in describing some of these families. A more general result is presented in \cite{KU}. We include a short proof of this special case for completeness.

 \begin{lemma}\label{lem: socle fits}
     A surjective degree-preserving homomorphism $\pi:A\to B$ between AG algebras of the same socle degree is an isomorphism.
 \end{lemma}
 \begin{proof}
Let $f\in \ker(\pi)$ and assume that $f\neq 0$. Then since $A$ is AG, there exists $g\in A$ so that $fg\in \Soc(A)$ and $fg\neq 0$. Suppose $\reg(A)=\reg(B)=d$. Since $\pi$ is degree-preserving and $A,B$ are AG, $\pi$ induces an isomorphism $\pi:A_d\xrightarrow{\cong} B_d$. Therefore, since $0\neq fg\in A_d$ we conclude  $0\neq \pi(fg)=\pi(f)\pi(g)$. This is a contradiction, since $\pi(f)=0$. Thus, $\ker(\pi)=0$ and so $\pi$ is an isomorphism.
 \end{proof}

 \subsection{New families of algebras satisfying the WLP}

 We are now ready to analyze the WLP for various families of AG algebras with binomial Macaulay dual generator. In terms of the notation introduced in \eqref{eq: binomial} our families will have one of the following properties
 \begin{enumerate}
    \item There is a variable in the support of $g$ that dominates, that is, its exponent exceeds the exponent of all the other variables combined. We analyze this case in \cref{family4}, where we note that in addition to satisfying WLP the algebras in this family also have a flat.
     \item The gcd $g$ of the two monomials in $F$ is a pure power of a single variable. We analyze this case in \cref{family1} and \cref{family2}.
     \item The support of the binomial $F/g$ contains at most three variables.  We analyze this case in \cref{family3} and \cref{family5}.
 \end{enumerate}

 Our first family concerns binomials $F$ for which the degree of a variable in the support of $g$ exceeds half of the degree of $F$.

\begin{theorem}\label{family4}
 We consider the binomial $$
F=X_1^{a_1}\cdots X_n^{a_n}(X_1^{b_1}\cdots X_r^{b_r}-X_{r+1}^{b_{r+1}}\cdots X_n^{b_n}), \quad 1\le r\le n-1,$$
where $\sum_{i=1}^r b_i=\sum_{i=r+1}^n b_i$.

\begin{itemize}
\item[(i)] Assume that $a_1-(a_2+\cdots +a_n)> b_{r+1}+\cdots +b_n.$ Then, $A_F$ has the WLP.
  \item[(ii)] Assume that $a_1-(a_2+\cdots +a_n)\ge  b_{r+1}+\cdots +b_n$ and $b_1>0$. Then,  $$NS_{A_F}\ge 3+a_1-(a_2+\cdots +a_n)-( b_{r+1}+\cdots +b_n).$$ 
\end{itemize}
\end{theorem}

\begin{proof}
(i) Denote $F = F_1-F_2$, where $F_i$ are the monomials
\begin{eqnarray*}
  F_1 &=& X_1^{a_1+b_1}\cdots X_r^{a_r+b_r}X_{r+1}^{a_{r+1}}\cdots X_n^{a_n}\\  F_2 &=& X_1^{a_1}\cdots X_r^{a_r}X_{r+1}^{a_{r+1}+b_{r+1}}\cdots X_n^{a_n+b_n}.  
\end{eqnarray*}  We set $d =\deg(F) = \sum_{i=1}^n a_i+\sum_{i=r+1}^n b_i$  and  $d' = \sum_{i=2}^n a_i+\sum_{i=r+1}^n b_i$. By assumption there inequality  $a_1>d'$ holds. 

We first consider the case $\deg(F)=2t+1$ in which we must have $a_1>t$ and $d'\leq t$. For each $0\le i\le d'$ consider 
\begin{align*}
\mathcal{A}_{(2,t)}^i = \{ x_1^{t-i}x_2^{\alpha_2}\cdots x_n^{\alpha_n}\mid  \alpha_2 + \cdots +\alpha_n = i \hspace{2mm}\text{and}\hspace{2mm}  &\alpha_2\le a_2,\dots ,\alpha_r\le a_r,\\
& \alpha_{r+1}\le a_{r+1}+b_{r+1},\dots ,\alpha_n\le a_n+b_n \}.
\end{align*}
Then $\mathcal{A}_{(2,t)}:=\bigcup_{i=0}^{d'}\mathcal{A}_{(2,t)}^i$ is the monomial basis for the monomial complete intersection algebra $A_{F_2}$ in degree $t$.  We expand  $\mathcal{A}_{(2,t)}$ to a monomial spanning  set for  $[A_F]_t$. To do so, consider each monomial $m_1$ of degree $t$ such that $m_1\notin \mathcal{A}_{(2,t)}$,  so $m_1\circ F_2=0$, and $m_1\circ F\neq 0$. Note that $m_1$ is of the form 
$$m_1=x_1^{t-\Bar{\beta}}x_2^{\beta_2}x_3^{\beta_3}\cdots x_r^{\beta_r}x_{r+1}^{\beta_{r+1}}\cdots x_n^{\beta_n}$$ 
such that $\beta_i\le a_i+b_i$ for $2\le i\le r$ and $\Bar{\beta}=\beta_2+\cdots +\beta_n$. 
We denote by $\mathcal{A}_{(1,t)}'$ the set of all such monomials for which $m_1\circ F\neq m_2\circ F$ for every $m_2\in \mathcal{A}_{(2,t)}$. 
Obviously, $\mathcal{B}_t:=\mathcal{A}_{(1,t)}'
\bigoplus \mathcal{A}_{(2,t)}$ spans $[A_F]_t$.

Order $\mathcal{B}_t$ in such a way that $\mathcal{A}_{(2,t)}$ and $\mathcal{A}_{(1,t)}'$ are ordered with respect to the reverse lexicographic order separately and elements in $\mathcal{A}_{(2,t)}$ are followed by elements in $\mathcal{A}_{(1,t)}'$. Consider $\Hess^t_F$ relative to $\mathcal{B}_t$ with this order. 

Denote by $M$ the submatrix of $\Hess^t_F$ corresponding to rows and columns in $\mathcal{A}_{(2,t)}$. We claim that $M$ is triangular with $X_1$ as anti-diagonal entries. We denote by $M_{(\mathcal{A}_{(2,t)}^i,\mathcal{A}_{(2,t)}^j)}$ the block of  $M$ with entries $uv\circ F$ such that $u\in \mathcal{A}_{(2,t)}^i,v\in \mathcal{A}_{(2,t)}^j$.
 Note that for every $u = x_1^{t-i}x_2^{\alpha_2}\cdots x_n^{\alpha_n}\in \mathcal{A}_{(2,t)}^i$ and $v = x_1^{t-j}x_2^{\alpha_2'}\cdots x_n^{\alpha_n'}\in \mathcal{A}_{(2,t)}^j$ we have 
 \[
 uv = x_1^{2t-i-j}x_2^{\alpha_2+\alpha_2'}\cdots x_n^{\alpha_n+\alpha_n'} \text{ such that } {\alpha_2+\alpha_2'}+\cdots + {\alpha_n+\alpha_n'} = i+j.\]
 If $j>d'-i$ or equivalently $i+j>d'$ we get $uv\circ F = 0$. Thus for every $j>d'-i$ the block $M_{(\mathcal{A}_{(2,t)}^i,\mathcal{A}_{(2,t)}^j)}$ is a zero block which means that all the entries of $M$ bellow the anti-diagonal blocks are zero. 
 
 \noindent Now we show that that for all $0\le i\le d'$, the anti-diagonal blocks of $M$, i.e. $M_{(\mathcal{A}_{(2,t)}^i,\mathcal{A}_{(2,t)}^{d'-i})}$, are square blocks for which the anti-diagonal entries are all equal to $X_1$ and entries bellow the anti-diagonal entries are all zero.
 For $u\in \mathcal{A}_{(2,t)}^i, v\in \mathcal{A}_{(2,t)}^{d'-i}$, we have  
 \[
 uv = x_1^{2t-d'}x_2^{\alpha_2+\alpha_2'}\cdots x_n^{\alpha_n+\alpha_n'} \quad \text{ where } \quad
 {\alpha_2+\alpha_2'}+\cdots + {\alpha_n+\alpha_n'} = d'.
 \]
 Since $\mathcal{A}_{(2,t)}$ is ordered in reverse lexicographic order the anti-diagonal entries of $M_{(\mathcal{A}_{(2,t)}^i,\mathcal{A}_{(2,t)}^{d'-i})}$ are equal to 
 $$x_1^{2t-d'}x_2^{a_2+b_2}\cdots x_r^{a_r+b_r}x_{r+1}^{a_{r+1}}\cdots x_n^{a_n}\circ F = X_1.
$$
The displayed equality holds because we have $2t-d'=a_1-1$.  On the other hand, the entries bellow the anti-diagonal entries of $M_{(\mathcal{A}_{(2,t)}^i,\mathcal{A}_{(2,t)}^{d'-i})}$ are equal to 
$$x_1^{2t-d'}x_2^{\alpha_2+\alpha'_2}\cdots x_n^{\alpha_n+\alpha'_n}\circ F = 0
$$ since $\alpha_n+\alpha'_n>a_n+b_n$. 

Let $\mathcal{A}_{(1,t)}$ be the monomial basis for the complete intersection algebra $A_{F_1}$. The same proof as above shows that the submatrix $N$ of $\Hess^t_F$ with respect to a basis of $[A_F]_t$ res\-tric\-ted to rows and columns in $\mathcal{A}_{(1,t)}$ is a triangular matrix having  $X_1$ as antidiagonal entries.

Now we claim that the submatrix of $N$ with rows and columns in $\mathcal{A}_{(1,t)}'$ has a similar shape. To show this we show that for $m_1,m_1'\in \mathcal{A}_{(1,t)}$ such that $m_1m_1'\circ F_1=X_1$, if  $m_1\in \mathcal{A}_{(1,t)}'$ then $m_1'\in \mathcal{A}_{(1,t)}'$. Suppose not. So either $m_1'\circ F_2\neq 0$ or $m_1'\circ F_1=m_2'\circ F_2$ for some $m_2'\in \mathcal{A}_{(2,t)}$.  

First assume that $m_1'\circ F_2\neq 0$ then $m_1'\circ \gcd(F_1,F_2)\neq 0$.  If we set
 \[
 m_1 = x_1^{\beta_1}x_2^{\beta_2}\cdots x_n^{\beta_n} \quad \text{ and } \quad
m_1' = x_1^{\gamma_1}x_2^{\gamma_2}\cdots x_n^{\gamma_n},
 \]
then for every $2\le i\le r$ we must have $\gamma_i\leq a_i$ and since $m_1'$ has degree $t$ and $a_1>t$ for $i=1$ we have $\gamma_1<a_1$. On the other hand, $m_1m_1'\circ F_1 = X_1$ forces $\beta_i+\gamma_i = a_i+b_i$ for every $2\le i\le r$.  This implies $\beta_i\geq b_i$ for every $2\le i\le r$. From these inequalities we obtain that for monomial $$m_2 = x_1^{a_1-\gamma_1-1}x_2^{\beta_2-b_2}\cdots x_r^{\beta_r-b_r}x_{r+1}^{\beta_{r+1}+b_{r+1}}\cdots x_n^{\beta_n+b_n}$$ we have $m_1\circ F = m_2\circ F$ which contradicts the assumption that $m_1\notin \mathcal{A}'_{(1,t)}$. Next if we assume $m_1'\circ F_1 = m_2'\circ F_2\neq 0$ then by acting by $m_1$ we obtain $m_1m_1'\circ F_1 = m_1m_2'\circ F_2\neq 0$ and this  means $m_1\circ F_2\neq 0$ which contradicts the assumption that $m_1\in \mathcal{A}_{(1,t)}'.$

We observe that the entries of the columns corresponding to $\mathcal{A}_{(1,t)}'$ and corresponding to row $\mathcal{A}_{(2,t)}$ is either zero or linear  in $X_2,\dots ,X_n$. By construction only $F_1$ contributes in columns corresponding to $\mathcal{A}_{(1,t)}'$ and since $\mathcal{A}_{(1,t)}'$ and $\mathcal{A}_{(2,t)}$ are disjoint sets there does not exist $m_2\in \mathcal{A}_{(2,t)}$ and  $m_1\in \mathcal{A}_{(1,t)}'$ such that $m_1m_2\circ F = m_1m_2\circ F_1 =  X_1$.

We conclude that $\Hess_F^t$ has the block form $\begin{pmatrix}
N&C\\
D&M\\
\end{pmatrix}$ where $N$ and $M$ are a triangular square matrices of size $|\mathcal{A}_{(1,t)}'|$ and $|\mathcal{A}_{(2,t)}|$ respectively with $X_1$ as anti-diagonal entries.  Moreover,  the entries of $D$ are in terms of $X_2,\dots ,X_n$.  This implies that $\Hess_F^t$ evaluated at $(1, 0 , \dots ,0)$ has non-zero determinant which allows us to use \cref{cor: hessian} and conclude that $x_1$ is a Lefschetz element for $A_F$ when $\deg(F)=2t+1$.

Now assuming $\deg(F) = 2t$ we get $a_1>t$ and $d'\leq t-1$.  We construct $\mathcal{B}_{t-1} = \mathcal{A}_{(1,t-1)}'\cup \mathcal{A}_{(2,t-1)}$ a spanning set for $[A_F]_{t-1}$ similar as above.  We show that $\mathcal{B}_t$ and $\mathcal{B}_{t-1}$ have the same number of elements.  First observe that since $d'\le t-1$ all elements in $\mathcal{A}_{(2,t)}$ are divisible by $x_1$ and moreover 
since we define $\mathcal{A}_{(2,t)}^i$ and $\mathcal{A}_{(2,t-1)}^i$ for $0\le i\le d'\le t-1$ multiplication by $x_1$ defines a bijection from $\mathcal{A}_{(2,t-1)}$ to $\mathcal{A}_{(2,t)}$.  Now note that if $m_1\in \mathcal{A}_{(1,t-1)}'$ satisfies  $m_1\circ F_1 \neq 0$, $m_1\circ F_2 =  0$ and  $m_1\circ F = m_2\circ F$ for some $m_2\in \mathcal{A}_{(2,t-1)}$ then there is a relation $x_1m_1\circ F = x_1m_2\circ F$  of degree $t$ implying $x_1m_1\in \mathcal{A}_{(1,t)}'$.  Since every monomial in the basis of $[A_F]_t$ is divisible by $x_1$ then dividing every monomial in $\mathcal{A}_{(1,t)}'$ gives a monomial in $\mathcal{A}_{(1,t-1)}'$.  This finishes the proof that multiplication by $x_1$ is a bijection from $\mathcal{B}_{t-1}$  to $\mathcal{B}_t$. Thus we have shown that $x_1$ is the Weak Lefschetz element for $A_F$ when $\deg(F)=2t$. 

\vskip 2mm

(ii) We proceed by induction on $t= a_1-(a_2+\cdots +a_n)-  (b_{r+1}+\cdots +b_n)$. Assume $t=0$ and consider the binomials $F$ and $G=X_1F$.
We first observe that the socle degree of $A_F$ is even and equal to $d=2a_1=2(a_2+\cdots +a_n+b_{r+1}+\cdots +b_n)$ while the socle degree of $A_G$ is odd and equal to $d+1=2a_1+1$. 
We define the K-vector spaces:
$$E_s=\langle x_1^{i_1}\cdots x_n^{i_n}\circ F \mid i_1+\cdots +i_n=s \rangle $$
and 
$$D_s=\langle x_1^{i_1}\cdots x_n^{i_n}\circ G \mid i_1+\cdots +i_n=s \rangle .$$
A straightforward computation gives us 
\begin{equation}\label{keyequality}
\dim E_s = \dim D_s \text{ for all } s\le a_1+1=\sum_{i=2}^na_i+\sum_{j=r+1}^nb_j +1.
\end{equation}
Therefore, we have
\begin{equation}\label{keyequality2}
    \HF_{A_G}(d+1-s)=\HF_{A_F}(d-s) \text{ for all } s\le a_1+1=\sum_{i=2}^na_i+\sum_{j=r+1}^nb_j +1.
\end{equation}
Since the socle degree of $A_G$ is odd and equal to $2a_1+1$ we  have $\HF_{A_G}(a_1)=\HF_{A_G}(a_1+1)$. Using the equalities (\ref{keyequality2}) we get  $\HF_{A_F}(a_1-1)=\HF_{A_F}(a_1)$ and being the socle degree of $A_F$ even and equal to $2a_1$ we conclude that $HF_{A_F}(a_1-1)=\HF_{A_F}(a_1)=\HF_{A_F}(a_1+1)$; i.e. $NS_{A_F}\ge 3$ which also implies that $NS_{A_G}\ge 4$. 

We now assume that the result is true for $t=\rho$ and we will prove for $t=\rho +1$.  We start with $$
F=X_1^{a_1}\cdots X_n^{a_n}(X_1^{b_1}\cdots X_r^{b_r}-X_{r+1}^{b_{r+1}}\cdots X_n^{b_n}), \quad 1\le r\le n-1,$$
where $\sum_{i=1}^r b_i=\sum_{i=r+1}^n b_i$, $b_1>0$ and $t = a_1-(a_2+\cdots +a_n)-  (b_{r+1}+\cdots +b_n)=\rho $.  We define $G=X_1F$ and  we denote by $d$ (resp. $d+1$) the socle degree of $A_F$ (resp. $A_G$).   Arguing as above we get
\begin{equation}\label{keyequality}
\dim \langle x_1^{i_1}\cdots x_n^{i_n}\circ F \mid i_1+\cdots +i_n=s \rangle  = \dim \langle x_1^{i_1}\cdots x_n^{i_n}\circ G \mid i_1+\cdots +i_n=s \rangle  \text{ for all } s\le a_1+1.
\end{equation}
and, hence, 
\begin{equation}\label{keyequality3}
    \HF_{A_G}(d+1-s)=\HF_{A_F}(d-s) \text{ for all } s\le a_1+1.
\end{equation}
which allow us to conclude that 
\[ NS_{A_G}=NS_{A_F}+1\ge (a_1+1)-(a_2+\cdots +a_n)-  (b_{r+1}+\cdots +b_n)=\rho +1,\]
which proves what we want.
\end{proof}

\vskip 2mm
Let us illustrate \cref{family4} with a concrete example.
\begin{example}
We consider the binomial $F=X^6Y^2(X^2Y-Z^3)$ and the associated codimension 3 Artinian Gorenstein algebra $A_F$.  The $h$-vector of $A_F$ is $$(1,3,6,10,12,12,12,12,10,6,3,1).$$ Therefore, we have: $S_{A_F}=12$ and $NS_{A_F}=4$. The monomial basis of $A_F$ in degree $5$, written in reverse lexicographic order, is
\[
\mathcal{B}=\{
\begin{small}
    x^{5}, x^{4}y, x^{4}z, x^{3}y^{2}, x^{3}y\,z, x^{3}z^{2}, x^{2}y^{2}z, x^{2}y\,z^{2}, x^{2}z^{3}, x\,y^{2}z^{2}, x\,y\,z^{3}, y^{2}z^{3}
\end{small}
      \}
\]
and with respect to this basis the Hessian matrix is
\newcommand{\rvline}{\hspace*{-\arraycolsep}\vline\hspace*{-\arraycolsep}}
\[
\Hess^5_{F}=
\begin{smallmatrix}\vphantom{0000000002\,z2\,y2\,x}\\ \vphantom{00000002\,z2\,y02\,x0}\\ \vphantom{0000002\,z2\,y02\,x00}\\ \vphantom{000002\,z002\,x000}\\ \vphantom{00002\,z2\,y02\,x0000}\\ \vphantom{0002\,z2\,y02\,x00000}\\ \vphantom{002\,z002\,x000000}\\ \vphantom{02\,z2\,y02\,x0000000}\\ \vphantom{02\,y02\,x00000000}\\ \vphantom{2\,z02\,x000000000}\\ \vphantom{2\,y2\,x0000000000}\\ \vphantom{2\,x00000000000}\\\end{smallmatrix}\left(\begin{smallmatrix}
      \vphantom{ }0&0&0&0&0&0&0&0&0&2\,z&2\,y&2\,x\\
      \vphantom{ }0&0&0&0&0&0&0&2\,z&2\,y&0&2\,x&0\\
      \vphantom{ }0&0&0&0&0&0&2\,z&2\,y&0&2\,x&0&0\\
      \vphantom{ }0&0&0&0&0&2\,z&0&0&2\,x&0&0&0\\
      \vphantom{ }0&0&0&0&2\,z&2\,y&0&2\,x&0&0&0&0\\
      \vphantom{ }0&0&0&2\,z&2\,y&0&2\,x&0&0&0&0&0\\
      \vphantom{ }0&0&2\,z&0&0&2\,x&0&0&0&0&0&0\\
      \vphantom{ }0&2\,z&2\,y&0&2\,x&0&0&0&0&0&0&0\\
      \vphantom{ }0&2\,y&0&2\,x&0&0&0&0&0&0&0&0\\
      \vphantom{ }2\,z&0&2\,x&0&0&0&0&0&0&0&0&0\\
      \vphantom{ }2\,y&2\,x&0&0&0&0&0&0&0&0&0&0\\
      \vphantom{ }2\,x&0&0&0&0&0&0&0&0&0&0&0\\
      \end{smallmatrix}\right)
      \]
and its determinant is $(2x)^{12}$. This shows that $x$ is a weak Lefschetz element on $A_F$ as predicted by  \cref{family4}.
\end{example}

The importance of studying the length  $NS_{A}$ of the flat portion in the Hilbert function of an AG algebra is given by the ability to transfer the Lefschetz property from an algebra having a flat to its quotients as illustrated in the following lemma.

\begin{lemma}\label{lem: flat}
   Let $F$ and $G$ be homogeneous polynomials in $S$ such that $F=p\circ G$ for some $p\in R$. If $A_G$ satisfies the weak Lefschetz property and $NS_{A_G} \geq \deg(p)+ 2$,  then $A_F$ satisfies the weak Lefschetz property and $NS_{A_F}\geq 2\left\lfloor \frac{NS_{A_G}}{2}\right\rfloor -\deg(p)$.
\end{lemma}
\begin{proof}
Set $\deg(G)=d$, $\deg(F)=e$, $\deg(p)=d-e$ and $t=\lfloor \frac{e}{2}\rfloor$. Since $F=p\circ G$, taking into account the definition of the dual algebras $A_F$ and $A_G$ there is a surjection $A_G\to A_F$. Applying the functor $\Hom_K(-,K)$ which turns $A_F$ and $A_G$ into their respective canonical modules produces an injection $A_F(e)\hookrightarrow A_G(d)$ or $A_F\hookrightarrow A_G(d-e)$ with cokernel $C$. Multiplication by a general linear form $\ell$ produces a commutative diagram
      \begin{equation}\label{CD1}
      \begin{tikzcd}
        0\arrow{r} &(A_F)_{i-1}\arrow{r}\arrow{d}{\times\ell}&(A_{G})_{i-1+d-e} \arrow{r} \arrow{d}{\times \ell} & C_{i-1+d-e} \arrow{r}\arrow{d}{\times \ell} &0\\
0\arrow{r} &(A_F)_{i}\arrow{r}{\times x_2}&(A_{G})_{i+d-e} \arrow{r}& C_{i+d-e}\arrow{r} &0.
      \end{tikzcd}    
\end{equation}
 We have that $H_{A_G}(i)=S_{A_G}$ for $ \lfloor \frac{d}{2}\rfloor -\left \lfloor\frac{NS_{A_G}}{2} \right\rfloor \leq i\leq \lfloor \frac{d}{2} \rfloor + \left \lfloor \frac{NS_{A_G}}{2} \right \rfloor$.  Thanks to  $A_G$ satisfying the WLP, this implies that the middle map in \eqref{CD1} is bijective for 
 \[
\left \lfloor \frac{d}{2} \right\rfloor - \left \lfloor \frac{NS_{A_G}}{2} \right \rfloor \leq i-1+d-e < i+d-e\leq \left \lfloor \frac{d}{2}\right \rfloor + \left \lfloor \frac{NS_{A_G}}{2} \right \rfloor.
 \]
 Thus by the snake lemma the leftmost vertical map in \eqref{CD1} is injective for 
 \begin{equation}\label{eq: range}
 e- \left\lceil \frac{d}{2} \right\rceil - \left \lfloor \frac{NS_{A_G}}{2} \right \rfloor +1 \leq i\leq e- \left \lceil \frac{d}{2} \right \rceil + \left \lfloor \frac{NS_{A_G}}{2} \right \rfloor.
 \end{equation}
         Since $NS_{A_G}\geq \deg(p)+2=d-e+2$ it follows that $i=\lfloor \frac{e}{2}\rfloor+1$ falls in the range in \eqref{eq: range}.  By \cref{mid map} we conclude that $A_F$ has WLP. Since the leftmost vertical map in \eqref{CD1} is surjective for $i\geq \lfloor \frac{e}{2}\rfloor +1$,   we conclude that  $H_{A_F}(i)=S_{A_F}$ for $ \lfloor \frac{e}{2}\rfloor  \leq i \leq e- \lceil \frac{d}{2} \rceil + \left \lfloor \frac{NS_{A_G}}{2} \right \rfloor$. Symmetrizing this range of indices about $\lfloor \frac{e}{2}\rfloor$ yields $NS_{A_F} \geq 2\left\lfloor \frac{NS_{A_G}}{2}\right\rfloor -\deg(p)$.
\end{proof}

The second family of binomial Macaulay dual generators we consider are characterized by the fact that  the gcd of the two monomials of $F$ is a pure power $X_i^a$ and  one of the monomials of $F/X_i^a$ is also a pure power $X_1^{b_1}$. We analyze separately the case when $i\neq 1$( \cref{family2}) and the case when $i=1$ (\cref{family1}).

 \begin{proposition}\label{family2} Fix an integer $i$, $2\le i \le n$. The algebra $A_F$ having Macaulay dual generator  the binomial 
 $$F=m_1-m_2=X_i^a(X_1^{b_1}-X_2^{b_2}X_3^{b_3}\cdots X_n^{b_n}) \qwith  \sum_{i=2}^nb_i=b_1,$$  has the WLP.
\end{proposition}

\begin{proof} Without loss of generality, we can assume $i=2$. We may also assume that at least one of $b_3, \ldots, b_n$ is nonzero, since the case when $b_3=\cdots=b_n=0$ is treated in \cref{family3}.
We distinguish 4 cases:
\begin{inparaenum}
\item  $a<b_1-1$,
 \item  $a> b_1$,
    \item $a= b_1$,
    \item  $a=b_1-1$.
\end{inparaenum}
\begin{enumerate}[leftmargin=2em]
\item Assume  $a<b_1-1$. In this case,  $\deg (\gcd(m_1,m_2))=a< \left \lfloor\frac{\deg (F)-1}{2} \right \rfloor$ and  we  apply  \cref{thm: conn sum criterion} to conclude that $A_F$ has the WLP. 

    \item Assume $a> b_1$. In this case, we apply \cref{family4} (i).

    \item Assume  $a=b_1$. 
Set $G=X_2F$.
By \cref{family4}  we know that $A_{G}$ has the WLP and $NS_{A_{G}}\ge 4$. since $A_{G}$ has  socle degree $2a+1$ (odd). Then, as $F=x_2\circ G$ , \cref{lem: flat} ensures $A_F$ has the WLP and $NS_{A_{F}}\ge 3$. 



    \item Assume  $a=b_1-1$. We consider $A_F$ and $A_{X_2F}$. By the previous case we know that $A_{X_2F}$ has the WLP and $NS_{A_{X_2F}}\ge 3$. Using once more \cref{lem: flat}, we conclude that $A_F$ also satisfies the WLP.     
\end{enumerate}
\end{proof}

 \begin{proposition}\label{family1} We consider the algebra $A_F$ having  Macaulay dual generator the binomial $$F=m_1-m_2=X_1^a(X_1^b-X_2^{b_2}X_3^{b_3}\cdots X_n^{b_n}) \qwith  \sum_{i=2}^nb_i=b.$$ If $K$ is a field of characteristic zero, then $A_F$ has the WLP. Moreover, if $a\ge b-1$, then $A_F$ is a complete intersection.
\end{proposition}

\begin{proof}
We distinguish 2 cases:

\begin{enumerate}[leftmargin=2em]
\item Assume $a<b-1$. In this case  $\deg (\gcd(m_1,m_2))=a< \left \lfloor\frac{\deg (F)-1}{2} \right \rfloor$ and we can apply \cref{thm: conn sum criterion} to conclude that $A_F$ has the WLP.

    \item Assume $a\ge b-1$. In this case, we consider the complete intersection ideal $$J=\langle x_2^{b_2+1},x_3^{b_3+1},\dots , x_n^{b_n+1}, x_1^{a+1}+x_1^{a-b+1}x_2^{b_2}\cdots x_n^{b_n}\rangle \subset \Ann(F).$$

This induces a canonical surjection $R/J\twoheadrightarrow A_F$. Since both $R/J$ and $A_F$ are AG algebras of the same socle degree $\reg(A_F)=\reg(R/J)=a+\sum_{i=2}^n b_i=a+b$, \cref{lem: socle fits} yields that $A_F\cong R/J$.
With respect to any monomial order $>$ such that $x_1>x_i$ for all $i\geq 2$, the initial ideal of $J$, $\init_>(J)$, is a monomial complete intersection 
\[
\init_>(J)=\langle x_1^{a+1}, x_2^{b_2+1},\ldots, x_n^{b_n+1}\rangle.
\]
Since $R/\init_>(J)$ has WLP in characteristic zero, we conclude, by \cite{Wiebe}, that $A_F=R/J$ has WLP.
\end{enumerate}
    \end{proof}

In the last family we analyze, the irreducible binomial factor of $F$ has small support, namely it can be written in term of two variables (\cref{family3}) or three variables (\cref{family5}). 

\begin{proposition}\label{family3}
If $K$ is a field of characteristic zero, the algebra $A_F$ having the binomial Macaulay dual generator $$F=X_1^{c_1}X_2^{c_2}X_3^{c_3}\cdots X_n^{c_n}(X_1^m-X_2^m) \qwith c_i>0,$$   has the SLP.
\end{proposition}
\begin{proof}
    For $$G_1=X_3^{c_3}\cdots X_n^{c_n}\qand G_2=X_1^{c_1}X_2^{c_2}(X_1^m-X_2^m) $$ consider the AG algebras
    $$A_{G_1}=K[x_3,\dots ,x_n]/\Ann(G_1) \qand A_{G_2}=K[x_1,x_2]/\Ann(G_2).$$
    We have $F =G_1G_2$ and therefore $A_F=A_{G_1}\otimes_K A_{G_2}$ using  the fact that $G_1$ and $G_2$
    are polynomials in disjoint sets of variables and \cite[Proposition 3.77]{LefscetzBook}.
    Now the SLP of $A_F$ follows from \cite[Theorem 3.10]{HW} and that $A_{G_1}$
    and $A_{G_2}$ have the SLP. Note that $A_{G_1}$ is a monomial complete intersection and, thus, it has the SLP in characteristic zero by \cite{Stanley}, and $A_{G_2}$ is an AG algebra of codimension two, and hence satisfies SLP in view of the fact that all codimension two Artinian algebras satisfy this property in characteristic zero; see \cite[Proposition 3.25]{HW}.
\end{proof}

Using the main result of our work \cite{ADFMMSV}, we can generalize the last result and prove:

\begin{proposition}\label{family5}
If $K$ is a field of characteristic zero, the algebra $A_F$ having the binomial Macaulay dual generator $$F=X_1^{c_1}X_2^{c_2}X_3^{c_3}\cdots X_n^{c_n}(m_1-m_2) $$
where $m_1$ and $m_2$ are monomials of the same degree in the variables $X_{i_1}$, $X_{i_2}$ and $X_{i_3}$, has the SLP.
\end{proposition}
\begin{proof} Without lose of generality, we can assume that $(i_1,i_2,i_3)=(1,2,3)$.
    For $$G_1=X_4^{c_4}\cdots X_n^{c_n}\qand G_2=X_1^{c_1}X_2^{c_2}X_3^{c_3}(m_1-m_2) $$ consider the AG algebras
    $$A_{G_1}=K[x_4,\dots ,x_n]/\Ann(G_1) \qand A_{G_2}=K[x_1,x_2,x_3]/\Ann(G_2).$$
    We have $F =G_1G_2$ and therefore $A_F=A_{G_1}\otimes_K A_{G_2}$ using  the fact that $G_1$ and $G_2$
    are polynomials in disjoint sets of variables and \cite[Proposition 3.77]{LefscetzBook}.
    Now the SLP of $A_F$ follows from \cite[Theorem 3.10]{HW} and that $A_{G_1}$
    and $A_{G_2}$ have the SLP. Note that $A_{G_1}$ is a monomial complete intersection and, thus, it has the SLP in characteristic zero by \cite{Stanley} while  $A_{G_2}$ is an AG algebra of codimension 3 with Macaulay dual generator a binomial, and hence by \cite[Theorem 3.45]{ADFMMSV} it satisfies SLP.
\end{proof}

\section{Open problems}\label{s: open}

In this last section we would like to formulate the open problems appearing in the introduction in the concrete case of AG algebras having binomial Macaulay dual generator. 

\begin{Oprobl}
Fix an integer $n\ge 3$ and consider a binomial $$
F=X_1^{a_1}\cdots X_n^{a_n}(X_1^{b_1}\cdots X_r^{b_r}-X_{r+1}^{b_{r+1}}\cdots X_n^{b_n}), \quad 1\le r\le n-1.$$
We ask:
\begin{enumerate}
    \item To determine in terms of $a_i$ and $b_j$ when $A_F$ is a complete intersection,
    \item To determine in terms of $a_i$ and $b_j$ when $A_F$ has WLP/SLP,
    \item To determine bounds for the Sperner number $S_{A_F}$ of $A_F$ and for $NS_{A_F}$,
    \item To determine an upper bound for the minimal number of generators of $A_F$,
     \item  To determine in terms of $a_i$ and $b_j$ when $A_F$ is a doubling,
    \item To determine a minimal free resolution of $A_F$.
\end{enumerate}
    
\end{Oprobl}

In \cite{ADFMMSV} the authors solved all the above problems in the codimension 3 case ($n=3$), while for arbitrary codimenson the problems are largely open, although partial results to some of them are given in previous sections of this paper.

\bibliographystyle{alpha} %
\bibliography{Proceedings}

\end{document}